\documentclass{article}
\usepackage{cmap}
\usepackage{amsthm}
\usepackage[cp1251]{inputenc}
\usepackage[T2A]{fontenc}
\usepackage{amsfonts}
\usepackage{amssymb}
\usepackage{amsmath}
\usepackage{color}
\usepackage{hyperref}
\usepackage[left=2.5cm,right=2.5cm, top=2.5cm,bottom=3cm]{geometry}
\newtheorem{lemma}{Lemma}
\newtheorem{prop}{Proposition}
\newtheorem*{prop*}{Proposition}
\newtheorem{coroll}{Corollary}
\newtheorem*{theorem}{Theorem}
\newtheorem*{HowieConj}{Howie conjecture}
\newtheorem*{maintheorem}{Main theorem}
\theoremstyle{definition}
\newtheorem*{defin}{Definition}
\newtheorem{quest}{Question}
\newtheorem{rem}{Remark}
\newtheorem*{example}{Example}
\let\tilde\widetilde
\let\hat\widehat

\begin{document}
\title{On $p$-nonsingular systems of equations over solvable groups\footnote{
This work was supported by
the Russian Science Foundation, 
project no. 22-11-00075}}
\author{Mikhail A. Mikheenko
\\{\small
Faculty of Mechanics and Mathematics
of Lomonosov Moscow State University
}
\\{\small
Moscow Center for
Fundamental and Applied Mathematics}\\
{\normalsize mamikheenko@mail.ru}
}
\date{}
\maketitle
\begin{abstract}
Any group that has a subnormal series,
in which all factors are abelian
and all except the last one are $p'$-torsion-free,
can be embedded into a group
with a subnormal series of the same length,
with the same properties and such that
any $p$-nonsingular system of equations
over this group is solvable in this group itself.
This helps us to prove that the minimal order of a metabelian
group, over which there is a unimodular equation that is
unsolvable in metabelian groups,
is $42$.
\end{abstract}

\setcounter{section}{-1}
\section{Introduction}
In this article $\mathbb Z_p$ for a prime $p$ denotes the field with $p$ elements, 
$(\mathbb Z_p)^*$ denotes the group of its invertible elements.
A cyclic group of order $n$ is denoted as $\langle a \rangle_n$,
$\langle b \rangle_n$, $\langle g \rangle_n$ etc.
An infinite cyclic group is denoted as
$\langle g \rangle_\infty$, $\langle x_1\rangle_\infty$ etc.
The Cartesian
(or unrestriscted direct)
product of groups $\{G_i\}_{i \in I}$
is denoted by $\prod\limits_{i \in I} G_i$,
whereas their direct (restricted) product is denoted as $\underset{i\in I}{\times} G_i$.
The group ring of a group
$G$ with a coefficient ring $R$ is denoted by $R G$.

Here we call a group $G$ an extension of a group $A$ by a group $B$ if
$G$ has the normal subgroup $A$ and the quotient group
$G/A$ is isomorphic to $B$.
For elements $g$ and $h$ of a group $G$
the expression $g^h$ denotes $h^{-1}gh$ and the expression
$g^{nh}$, where $n$ is an integer (may be negative),
denotes $h^{-1}g^nh$.
Also, $[g,h]$ denotes $g^{-1}h^{-1}gh$.
The Cartesian (or unrestricted) wreath product $A \bar \wr B$
is viewed as
$\left(\prod\limits_{b \in B} A_b\right)\leftthreetimes B$,
in which $A_b$ are copies of $A$
and $B$ acts on $\prod\limits_{b \in B} A_b$ such that
$\left((a_b)_{b\in B}\right)^{b_1}=
(a_b)_{bb_1\in B}=
(a_{bb_1^{-1}})_{b\in B}$
if $b_1 \in B$ and
$a_b \in A_b$ for any $b\in B$.

The paper is devoted to equations and systems of equations over groups.

\begin{defin}
Let $G$ be a group.
An equation in variables $x_1,\ldots,x_n$
over $G$ is an expression
$w(x_1,\ldots,x_n)=1$, where $w$ is an element of the free product
$G*F(x_1,\ldots,x_n)$,
in which $F(x_1,\ldots,x_n)$ is the free group
with basis $x_1,\ldots,x_n$.

The equations $w(x_1,\ldots,x_n)=1$ is {\it solvable} in the group
$\tilde G$ if $\tilde G \supset G$and $\tilde G$ contains a
solution of this equations (i.e. there are elements
$\tilde g_1,\ldots,\tilde g_n \in \tilde G$ such that $w(\tilde g_1,\ldots, \tilde g_n)=1$).
The group $\tilde G$ is called a {\it solution group}.
Equivalently, $w=1$ is solvable in $\tilde G$
if there is a homomorphism $G*F(x_1,\ldots,x_n) \to \tilde G$
which is injective on $G$ and sends $w$ to $1$.

If $w=1$ is solvable in some group $\tilde G$,
we say that the equation $w=1$ over $G$ is {\it solvable}.

The solvability of a (finite or infinite) system of equations
(possibly in an infinite set of variables)
over a group is defined likewise.
\end{defin}

It is easy to see that the system 
$\{w_j=1\}_{j \in J}$ of equations in variables $\{x_i\}_{i \in I}=X$
over $G$ is solvable if and only if
$G \cap \langle \langle W \rangle \rangle=\{1\}$
in $G*F\left(X\right)$, where
$W$ is the set $\{w_j\}_{j \in J}$
and $\langle \langle W \rangle \rangle$ is its normal closure
in $G*F\left(X\right)$.
It follows that,
if every system's finite subsystem
is solvable, then
the systems itself is solvable.

\begin{defin}
Let $\{w_j = 1\}_{j \in J}$ be
a system of equations in variables $\{x_i\}_{i \in I}=X$
over a group $G$.
Consider the free $\mathbb Z$-module
$\displaystyle \sum\limits_{i \in I}\mathbb Z \cdot x_i$
with basis $X$. It can also be viewed as a free abelian group in additive notation.
There is a homomorphism (which sends group multiplication to addition in module)
$G*F\left(X\right) \to \displaystyle \sum\limits_{i \in I}\mathbb Z \cdot x_i$
which sends $G$ to 0 and $x_i \in X$ to $x_i$.
Let us call this homomorphism a trivialization.

Let $m_j$ be the image of $w_j$ under trivialization.
$m_j$ can be viewed as a finitely supported
(i. e. having only a finite number of non-zero coordinates)
row of exponent sums of the variables
$x_i$ in $w_j$.
The system $\{w_j=1\}$ is called {\it non-singular},
if elements $m_j$ are independent over $\mathbb Z$,
i. e. i.e. there is no combination of these rows with coefficients from
$\mathbb Z$ which is equal to zero, except the combination
where all the coefficients are zero
(equivalently: if they are linearly independent over
$\mathbb Q$ as elements of the vector space
$\displaystyle\sum\limits_{i \in I}\mathbb Q \cdot x_i$).

For a prime number $p$
consider likewise a
$p$-trivialization,
i. e. the homomorphism
$G*F\left(X\right) \to \displaystyle \sum\limits_{i \in I}\mathbb Z_p \cdot x_i$.
The system is called {\it $p$-nonsingular}
if the images of the words $w_j$ under $p$-trivialization
(i.e. rows of exponent sums of variables in $w_j$
modulo $p$)
are linearly independent over
$\mathbb Z_p$.
The system is called {\it unimodular}
if it is $p$-nonsingular for every prime $p$.

In particular, an equation $w(x)=1$ in one variable is
non-singular ($p$-nonsingular, unimodular) if the exponent sum of $x$ in $w$
is non-zero (is not divisible by $p$, is equal to $\pm 1$ respectively).
\end{defin}

Clearly, if a system of equations is $p$-nonsingular for some prime $p$,
then it is non-singular as well. Also, a unimodular system
is the same as a $p$-nonsingular system for every prime $p$, therefore,
in particular,
unimodular systems are non-singular.

\begin{example}
The system of equations
$$
\left\{
\begin{array}{l}
[x,y]x^2 g_1 y^{-3}=1,  \\ \relax
[y,z]z=1, \\
xg_2yg_3z=1 \\
\end{array}
\right.
$$
has the following exponent sums rows organized into a matrix:
$$
\begin{pmatrix}
2 & -3 & 0\\
0 & 0 & 1\\
1 & 1 & 1
\end{pmatrix}.
$$
The determinant of this matrix is $-5$, so the system
is non-singular, $2$-nonsingular, $3$-nonsingular,
but is $5$-singular, so it is not unimodular.
\end{example}

Non-singular systems of equations over
groups from many classes are solvable.
For example, one of the classical results on equations over groups
concerns finite groups:

\begin{theorem}[\cite{GR62}]
A finite non-singular system of equations over a finite group
$G$ is solvable. Moreover,
a solution group can be chosen to be finite.
\end{theorem}
This theorem was repeatedly generalized, see
\cite{P08,KT17,T18,NT22,KM23}.

The following conjecture is as yet
neither proved nor disproved:

\begin{HowieConj}[\cite{How81}]
Any non-singular system of equations over any group is solvable.
\end{HowieConj}

As any system of equations over a group is solvable
if and only if every finite subsystem of this system is solvable,
it suffices to consider only finite systems
in order to check whether the conjecture is true.
For example, the result from \cite{GR62} states that the conjecture is true
for finite groups.
However, a solution group for an infinite system may not be finite.

Besides finite groups, this conjecture
is true for locally indicable groups as well.

\begin{defin}
A {\it locally indicable group} is a group
whose any non-trivial finitely generated subgroup
admits a surjective homomorphism onto $\langle g \rangle_\infty$.
\end{defin}

\begin{defin}
A {\it locally $p$-indicable group} for a prime $p$ is a group
whose any non-trivial finitely generated subgroup
admits a surjective homomorphism onto  $\langle g \rangle_p$.
\end{defin}

Note that locally indicable groups are
locally $p$-indicable for any prime $p$.
Also, it is easy to see that, if a group is locally $p$-indicable
for an arbitrarily large prime $p$, then it is locally indicable.
This can be seen by considering a quotient group of a finitely
generated subgroup by the commutator subgroup of this subgroup.

\begin{theorem}[\cite{How81}]
A finite non-singular system of equations over a locally indicable group
is solvable.
\end{theorem}

There is a similar result for $p$-nonsingular systems.

\begin{theorem}[\cite{Kr85}]
Let $p$ be a prime.
A finite $p$-nonsingular system of equations over a locally $p$-indicable group
is solvable.
\end{theorem}

The following result shows that more is known about unimodular
equations than about non-singular ones.

\begin{theorem}[\cite{K93}]
Any unimodular equation over a torsion-free group
is solvable.
\end{theorem}
This result was generalized in \cite{K06}.
It is unknown if an anologous result is true
for non-singular equations.

Some results are connected not only with the solvability of
systems of equations over groups but also with
the solvability of systems in groups of the same class.
Aside from the already mentioned result from \cite{GR62} for the class of finite groups
there is a long-ago obtained result for the class of nilpotent groups.
\begin{defin}
Let $p$ be a prime.
A group $G$ is called {\it $p'$-torsion-free} if any non-identity element of $G$
has either infinite order or order equal to a power of $p$.
\end{defin}

\begin{theorem}[\cite{Sh67}]
If
\begin{itemize}
\item $G$ is a divisible nilpotent torsion-free group and there is
a finite non-singular system of equations over $G$;
\item $G$ is nilpotent and
there is a finite unimodular system of equations over $G$;
\item or $G$ is a $p'$-torsion-free nilpotent group in which
any element is a $q$th power for any prime
number $q$ not equal to $p$
(so $G$ is $p'$-divisible)
and there is a finite $p$-nonsingular system of equations over $G$;
\end{itemize}
then the system is solvable in $G$, and the solution in $G$ is unique.
\end{theorem}

This implies that, for example, any non-singular system
over a nilpotent torsion-free group is solvable in
the divisible hull of the group,
which is a nilpotent group of the same nilpotency class.

\cite{KMR24} studied a similar question for the class
of solvable groups: when a non-singular system
over a solvable group has a solution in a solvable group as well
and what can be said about derived length of the solution group?

\begin{prop*}[\cite{KMR24}]
There exists a metabelian group $G$ with a unimodular equation $w(x)=1$ over it such
that the equation $w(x)=1$ has no solutions in metabelian groups.

Moreover, $G$ can be chosen to be:
\begin{itemize}
\item either finite (of order $42$)
\item or torsion-free and such that any unimodular equation
(and even any non-singular system) over $G$
has a solution in the class of solvable groups of derived length $3$.
\end{itemize}
\end{prop*}

\begin{theorem}[\cite{KMR24}]
Let a group $G$ have a subnormal series
$$G=G_1\triangleright G_2 \triangleright \ldots \triangleright G_{n} \triangleright G_{n+1}=\{1\}$$
in which all factors are abelian and all factors except the last one are torsion-free.
Then any (finite or infinite)
non-singular system of equations over $G$
has a solution in some group $\tilde G \supset G$
which also has a subnormal series
$$\tilde G= \tilde G_1\triangleright  \tilde G_2 \triangleright \ldots \triangleright  \tilde G_{n}
\triangleright  \tilde G_{n+1}=\{1\}$$
with abelian factors, and all factors of the series of $\tilde G$, except for the last factor, are torsion-free.
\end{theorem}

Section \ref{mainsection} of the present article
contains the proof to a similar result for $p$-nonsingular systems.

\begin{maintheorem}
Let $p$ be a prime number.
Let a group $G$ have a subnormal series
$$
G=G_1 \triangleright G_2 \triangleright  \ldots \triangleright G_n \triangleright G_{n+1} =\{1\}
$$
with factors
$B_1, B_2,\dots,B_n$, where $B_i=G_i/G_{i+1}$ are abelian groups
and
all the factors $B_i$ except
$B_n$ are $p'$-torsion-free.

Then there is a group $\hat G \supset G$ with a subnormal series
$$
\hat G=\hat G_1 \triangleright \hat G_2 \triangleright \ldots \triangleright \hat G_n \triangleright \hat G_{n+1} =\{1\}
$$
with the same properties
(i.e. its length is $n$, factors are abelian and all factors
except the last one are $p'$-torsion-free)
such that any $p$-nonsingular system of equations over $G$
has a solution in $\hat G$.
Moreover, every $p$-nonsingular system over $\hat G$
is solvable in $\hat G$ itself.
\end{maintheorem}
Arguments of Section \ref{mainsection} allow us to strengthen
the theorem from \cite{KMR24}.

Lemma \ref{Lemma1} from Section \ref{ringsection}
plays an important role in the proof of the main theorem.
Also, the method of the proof is based on methods that were developed in
\cite{NT22,KM23,KMR24}.

In Section \ref{countersection} the main theorem helps us to
prove the following fact.

\begin{prop}\label{counterprop}
The metabelian group
$
\begin{pmatrix}
1 & \mathbb Z_7 \\
0 & (\mathbb Z_7)^* \\
\end{pmatrix}
$
of order $42$ from \cite{KMR24} is a minimal (by order)
example of a metabelian group over which there exists
a unimodular equation that has no
solutions in metabelian groups.
\end{prop}

Some open questions on systems of equations
over solvable groups
are formulated in Section \ref{questsection}.

The author thanks the Theoretical Physics and Mathematics
Advancement Foundation
``BASIS''.
The author also thanks Anton A. Klyachko
for valuable advices and remarks.
The author thanks an anonymous referee,
who, with other useful comments,
suggested a shorter proof for Section
\ref{countersection}.

\section{Images of rows of group ring elements}\label{ringsection}

\begin{prop}\label{ringprop}
Let $R$ be an associative ring with unity of characteristic $p$.
Let $M \in R[x]/(x^{p^k}-1)$ be an element such that
$f(M)$ is not a left (right) zero divisor, where
$f\colon R[x]/(x^{p^k}-1) \to R$ is the natural ring homomorphism with
$f(x)=1.$
Then $M$ itself is not a left (right) zero divisor.
\end{prop}

\begin{proof}
We consider the ``left'' case.

Consider $M$ as a polynomial in $(x-1)$ with coefficients from $R$:
$$
M=M_0+M_1(x-1)+M_2(x-1)^2+\ldots+M_{p^k-1}(x-1)^{p^k-1}.
$$
$(x-1)^{p^k}=x^{p^k}-1=0$ because the characteristic of $R$ is $p$.
By the definition of $f$, $M_0$ is equal to $f(M)$,
hence, $M_0$ is not a left zero divisor.

Now assume that $M$ is a left zero divisor. This means that there is a non-zero polynomial
$$B=B_0+B_1(x-1)+\ldots+B_{p^k-1}(x-1)^{p^k-1}$$
such that $MB=0$.
Consider $MB$ as a polynomial in $(x-1)$ as well.
Look at the constant term of $MB$. On the one hand,
it is equal to $M_0B_0$,
on the other hand, it must be zero.
$M_0$ is not a left zero divisor, therefore, $B_0=0$.
Then look at the linear term of $MB$.
As $B_0$ equals $0$, the linear term is equal to
$M_0B_1(x-1)$, while it must be equal to $0$ as well as the constant term.
And so on. Eventually we get that all the coefficients $B_i$ are equal to $0$.
This contradicts the condition $B\neq 0$, thus, the assumption that $M$
is a left zero divisor is false.

The ``right'' case can be treated similarly.
\end{proof}

\begin{coroll}\label{coroll1}
Let $R$ be an associative ring with unity of characteristic $p$.

Let $M \in R[x_1,\ldots,x_l]/(x_1^{p^{k_1}}-1,\ldots, x_l^{p^{k_l}}-1)$
be an element such that
$f(M)$ is not a left (right) zero divisor, where
$f\colon R[x_1,\ldots,x_l]/(x_1^{p^{k_1}}-1,\ldots, x_l^{p^{k_l}}-1) \to R$
is the natural ring homomorphism with
$f(x_i)=1.$
Then $M$ itself is not a left (right) zero divisor.
\end{coroll}
\begin{proof}
Let us use induction by $l$.

If $l=1$, then this is exactly Proposition \ref{ringprop}.

Now show that, if the corollary is true for $l$ variables,
then it is true for $l+1$ variables as well.

Note that
$$
R[x_1,\ldots,x_l,x_{+1}]/(x_1^{p^{k_1}}-1,\ldots, x_l^{p^{k_l}}-1,
x_{l+1}^{p^{k_{l+1}}}-1)
=
$$
$$
=
\left(R[x_1,\ldots,x_l]/(x_1^{p^{k_1}}-1,\ldots, x_l^{p^{k_l}}-1)\right)
[x_{l+1}]/(x_{l+1}^{p^{k_{l+1}}}-1).
$$
Denote
$R[x_1,\ldots,x_l]/(x_1^{p^{k_1}}-1,\ldots, x_l^{p^{k_l}}-1)$
as $Q$. The ring $Q$ has the characteristic $p$ as well.
Look at the image $h(M)$ of $M$ under the natural homomorphism
$h\colon Q[x_{l+1}]/(x_{l+1}^{p^{k_{l+1}}}-1) \to Q$, $h(x_{l+1})=1$.
This image is not a left (right) zero divisor by induction.
Then, by Proposition \ref{ringprop}, the element $M$
is also not a left (right) zero divisor.
\end{proof}

Now let $P$ be a finitely generated abelian $p$-group.
It can be represented as
$P= \langle g_1\rangle_{p^{k_1}}\times\ldots\times \langle g_n\rangle_{p^{k_l}}$.
Then the group algebra $\mathbb Z_p P$ can be represented as
$\mathbb Z_p P=\mathbb Z_p [x_1,\ldots,x_l]/(x_1^{p^{k_1}}-1,\ldots, x_l^{p^{k_l}}-1)$.
And the ring of $n\times n$ matrices over $\mathbb Z_p P$
is
$M_n(\mathbb Z_p) [x_1,\ldots,x_l]/(x_1^{p^{k_1}}-1,\ldots, x_l^{p^{k_l}}-1)$.
Knowing that a non-singular
matrix over a field is not a zero divisor, we get the following fact.

\begin{coroll}\label{coroll2}
Let $M\in M_n(\mathbb Z_p P)$
be an $n \times n$ matrix over the group algebra, where $P$
is a finitely generated abelian $p$-group.
Suppose that the image of $M$ under the natural ring homomorphism
$\varepsilon\colon M_n(\mathbb Z_p P) \to M_n(\mathbb Z_p)$,
which sends the elements of the group $P$ to the unity of $\mathbb Z_p$,
is a non-singular matrix.

Then $M$ is not a zero divisor.
In particular, the rows of $M$ are independent over $\mathbb Z_p P$.
\end{coroll}

The statement of Corollary
\ref{coroll2}
is also true for $P\times A$,
where $P$ and $A$ are finitely generated abelian groups,
$P$ is a $p$-group and $A$ is torsion-free.

\begin{coroll}\label{coroll3}
Let
$M\in M_n(\mathbb Z_p \left(P\times A\right))$
be an $n \times n$ matrix over the group algebra, where 
$P$ is a finitely generated abelian $p$-group and
$A$ is a finitely generated torsion-free abelian group.
Suppose that the image of $M$ under the natural ring homomorphism
$\varepsilon\colon M_n(\mathbb Z_p \left(P\times A\right)) \to M_n(\mathbb Z_p)$,
which sends the elements of $P\times A$ to the unity of $\mathbb Z_p$,
is a non-singular matrix.

Then $M$ is not a zero divisor.
In particular, the rows of $M$ are independent over $\mathbb Z_p \left(P\times A\right)$.
\end{coroll}
\begin{proof}
Look at the image of $M$ under the ring homomorphism
$f\colon M_n(\mathbb Z_p \left(P\times A\right)) \to M_n(\mathbb Z_p A)$,
which sends the elements of $P$ to the unity of $\mathbb Z_p$
and does not change the elements of $A$.
The matrix $f(M)$ is not a zero divisor.

Indeed, $\mathbb Z_p A$
can be embedded into its field of fractions,
as $\mathbb Z_p A$ is an integral domain.
Therefore, the matrix $f(M)$
not being a zero divisor
is equivalent to the non-singularity of $f(M)$
over the field of fractions,
i.e. to it having non-zero determinant.
And the determinant of $f(M)$ is not equal to zero,
as by the condition the determinant of $\varepsilon(M)=h\left(f(M)\right)$
is not equal to zero, where $h\colon M_n(\mathbb Z_p A) \to M_n(\mathbb Z_p)$ 
is the natural ring homomorphism which
sends the elements of $A$ to the unity of $\mathbb Z_p$.

Futher note that
$M_n(\mathbb Z_p \left(P\times A\right))$
is naturally isomorphic to
$M_n(\mathbb Z_p A) [x_1,\ldots,x_l]/(x_1^{p^{k_1}}-1,\ldots, x_l^{p^{k_l}}-1)$.
Then, by Corollary \ref{coroll1}, in which
$R$ is $M_n(\mathbb Z_p A)$,
we get that $M$ is not a zero divisor as
$f(M)$ is not a zero divisor.
\end{proof}

The structure $P\times A$ (where $P$ and $A$ are from Corollary \ref{coroll3})
is the structure of every finitely generated $p'$-torsion-free abelian group.
We can now go from the finitely generated case to the infintely generated one.

\begin{coroll}\label{coroll4}
Let $M\in M_n(\mathbb Z_p D)$
be an $n \times n$ matrix over the group algebra, where 
$D$ is a $p'$-torsion-free abelian group.
Suppose that the image of $M$ under the natural ring homomorphism
$\varepsilon\colon M_n(\mathbb Z_p D) \to M_n(\mathbb Z_p)$
is a non-singular matrix.

Then $M$ is not a zero divisor.
In particular, the rows of $M$ are independent over $\mathbb Z_p D$.
\end{coroll}

\begin{proof}
Assume that this is not true.
It means that there is a non-zero matrix $B$
over $\mathbb Z_p D$ such that
either $BM=0$ or $MB=0$.

Look at all the elements of  $B$ and $M$.
There is a finite number of these elements.
Each of them contains a finite number of the elements of $D$
with a non-zero coefficient.
Hence, there is only a finite number of the
elements of $D$ which are contained in these matrices
with a non-zero coefficient.

Consider the subgroup $H$ generated by those elements.
It is a finitely generated $p'$-torsion-free abelian group.
The elements of $B$ and $M$
are contained in the corresponding subalgebra $\mathbb Z_p H$.
Therefore, we get that $M$ is a zero divisor
in a matrix ring over the group algebra of a
finitely generated $p'$-torsion-free abelian group,
which is prohibited by Corollary \ref{coroll3}.

It means that the assumption of $M$ being a zero divisor is false.
\end{proof}

Corollary \ref{coroll4}
can be generalized to the infinite rows case.
\begin{lemma} \label{Lemma1}
Let $D$ be a $p'$-torsion-free abelian group
and
$\{m_i\}_{i\in I}$ be elements of a free module
over
$\mathbb Z_p D$ (i.e. finitely supported
rows of elements of this group algebra).
Consider the natural ring homomorphism
$f \colon \mathbb Z_p D \to \mathbb Z_p$
which sends the elements of the group $D$ to
the unity of $\mathbb Z_p$.
This mapping naturally defines
a mapping from the free module over $\mathbb Z_p D$
to the vector space over $\mathbb Z_p$.

Now let the system of rows $\{f(m_i)\}_{i\in I}$ be
linearly independent over $\mathbb Z_p$.
Then the system $\{m_i\}_{i\in I}$ is independent over $\mathbb Z_p D$.
\end{lemma}
\begin{proof}
Assume that $\{m_i\}_{i\in I}$ is dependent over $\mathbb Z_p D$.
It means that some finite subsystem of these rows is dependent.
Without loss of generality, suppose that this subsystem consists
of $m_1, m_2,\ldots, m_n$.

By the condition of the lemma $f(m_1),f(m_2),\ldots f(m_n)$
are linearly independent over $\mathbb Z_p$.
So if we write the rows $f(m_1),f(m_2),\ldots f(m_n)$ one above another
and remove all the zero columns (after this only
a finite number of columns remains),
we get a matrix with linearly independent rows.
In particular, this matrix has a non-singular square submatrix of size $n\times n$.

Now consider the corresponding submatrix for $m_1,m_2,\ldots m_n$.
By Corollary \ref{coroll4} the rows of this submatrix are independent over
$\mathbb Z_p D$.
So the rows $m_1,m_2,\ldots m_n$ are
independent over $\mathbb Z_p D$.
We get a contradiction, thus, the assumption
of the dependence of $\{m_i\}_{i\in I}$
is false.
\end{proof}
\section{The proof of the main theorem}\label{mainsection}
Let us use induction by $n$.

If $n=1$ (i.e. the group $G$ is abelian), then the divisible hull of $G$
can be taken as $\hat G$.
Indeed, suppose that
$\{w_j=1\}_{j \in J}$ is a non-singular (even not necessarily $p$-nonsingular)
system of equations over a divisible
abelian group $H$ 
in variables $\{x_i\}_{i \in I}$.

The group $H$ is naturally embedded into
$$
R=\left( H \times \left(\underset{i\in I}{\times}\langle x_i \rangle_\infty\right) \right)
\biggl/\Bigl\langle \{w_j\}_{j\in J}\Bigr\rangle.
$$

Indeed, to show that the natural mapping is
injective it suffices to show that
$H \cap \Bigl\langle \{w_j\}_{j\in J}\Bigr\rangle=\{1\}$ in
$H \times \left(\underset{i\in I}{\times}\langle x_i \rangle_\infty\right)$.
Consider an element
$w_j \in H \times \left(\underset{i\in I}{\times}\langle x_i \rangle_\infty\right)$
and its component in
$\underset{i\in I}{\times}\langle x_i \rangle_\infty$.
The group $\underset{i\in I}{\times}\langle x_i \rangle_\infty$
can be rewritten in additive notation to obtain
the $\mathbb Z$-module $\displaystyle \sum\limits_{i \in I}\mathbb Z \cdot x_i$.
The component of $w_j$ in this notation is equal to $m_j$ ---
the image of $w_j$ under trivialization.
Then
$\left(\underset{i\in I}{\times}\langle x_i \rangle_\infty\right)$-component
of an arbitrary element
$w_{j_1}^{n_1}\ldots w_{j_s}^{n_s}$ from
$\Bigl\langle \{w_j\}_{j\in J}\Bigr\rangle$
in additive notation is
$n_1m_{j_1} + \ldots + n_s m_{j_s}$.
If element $w_{j_1}^{n_1}\ldots w_{j_s}^{n_s}$ lies in $H$,
then its
$\left(\underset{i\in I}{\times}\langle x_i \rangle_\infty\right)$-component
is $1$ in multiplicative notation, i. e. $0$ in additive notation.
However, as the system $\{w_j=1\}$ is non-singular,
$m_j$ are independent over $\mathbb Z$,
hence $n_1m_{j_1} + \ldots + n_s m_{j_s}=0$
only if $n_1=\ldots=n_s=0$.
It means that, if the element
$w_{j_1}^{n_1}\ldots w_{j_s}^{n_s}$ lies in $H$,
then it must be equal to $1$ (in multiplicative notation).
This shows that
$H \cap \Bigl\langle \{w_j\}_{j\in J}\Bigr\rangle=\{1\}$
in
$H \times \left(\underset{i\in I}{\times}\langle x_i \rangle_\infty\right)$
and also shows the injectivity of the natural mapping
$H\to R$.

Clearly, $R$ contains a solution $\{x_i\}_{i \in I}$ of $\{w_j=1\}$.
Since $H$ is divisible, it is a direct factor of $R$.
The image of
$\{x_i\}_{i \in I}$ under the projection onto $H$ 
forms a solution of the system in $H$.

Now assume that $n>1$.
Embed $G$ into the Cartesian
wreath product
$G_2 \bar \wr B_1$
by the Kaloujnine-Krasner theorem
\cite{KK51}.
After this we can embed $B_1$ into its divisible hull $\hat B_1$
and embed $G_2$ (by induction)
into a group $\tilde G_2$ with a required subnormal series
of length $n-1$ and such that every $p$-nonsingular 
system of equations over $\tilde G_2$ has a solution in $\tilde G_2$
itself.
Thus, we embed $G_2 \bar \wr B_1$ into
the wreath product $\tilde G_2 \bar \wr \hat B_1$.
Note that the divisible hull $\hat B_1$ of
the $p'$-torsion-free abelian group $B_1$
is $p'$-torsion-free as well
(to understand this, it suffices to consider the torsion-free component
and the $p$-primary component
of $\hat B_1$ and to see that the projection of $\hat B_1$ onto the product of
these two components is injective on $B_1\subset \hat B_1$).

Now we need the following lemma.

\begin{lemma}\label{Lemma2}
Suppose that $\hat H$ is a group such that
every $p$-nonsingular system over it
has a solution in $\hat H$ itself.
Also suppose that $\hat B$ is a divisible $p'$-torsion-free abelian group.

Then every $p$-nonsingular system of equations over
the Cartesian wreath product
$\hat H \bar \wr \hat B=\left(\prod\limits_{b \in \hat B}\hat H_b\right)\leftthreetimes \hat B$
has a solution in $\hat H \bar \wr \hat B$ itself.
\end{lemma}
\begin{proof}
Suppose that $\{v_j=1\}_{j\in J}$ is
a $p$-nonsingular system of equations
over $\hat H \bar \wr \hat B$
in variables $\{x_i\}_{i\in I}$.
Consider the image of this system
under the natural mapping of coefficients $\hat H \bar \wr \hat B \to \hat B$.
This image is a $p$-nonsingular system of equations
over the divisible abelian group $\hat B$,
therefore, it has a solution in $\hat B$ itself.

It means that we can change
the variables
in the initial system $\{v_j=1\}_{j\in J}$
over $\hat H \bar \wr \hat B$ such
that all the words $v_j$
become words in the alphabet
$\{x_i^{\pm b}\}_{i\in I, b\in \hat B} \bigsqcup
\prod\limits_{b \in \hat B}\hat H_b$.
From now
we assume that all the words
$v_j$
were originally in this alphabet.
Denote
$\prod\limits_{b \in \hat B}\hat H_b$ by $C$.

Let us introduce new variables
$\{x_{ib}\}_{i\in I, b\in \hat B}$
and rewrite the words 
$v_j$
of coefficients from $C$ and the variables $\{x^{\pm 1}_i\}$
as words $w_j$ of coefficients from $C$ and the variables 
$\{x^{\pm 1}_{ib}\}$ such
that for every $j \in J$
$v_j$ is $w_j$ in which
$x^{\pm 1}_{ib}$ is replaced by
$x_i^{\pm b}$ for all $i\in I$ and $b \in \hat B$.
In other words,
$$
v_j = w_j\left(c_{(j,1)},\ldots, c_{(j,k_j)}, x_{i_{(j,1)}}^{b_{(j,1)}},
\ldots, x_{i_{(j,l_j)}}^{b_{(j,l_j)}}\right).
$$

Search for a solution of $\{v_j=1\}_{j \in J}$ among elements of $C$.
Let $\{\tilde x_i\}_{i\in I}$ be a set of elements from $C$.
By $[\tilde x_i]_b$ denote the coordinate of
$\tilde x_i$, correspoding to
the factor $\hat H_b$.
Note that for every
$d \in \hat B$
$$
[\tilde x_i^{d}]_b = [\tilde x_i]_{bd^{-1}}.
$$

For every $i \in I$ replace $x_i$ in $v_j$
with $\tilde x_i$. Denote the result (which is an element of $C$) by
$\tilde v_j$ and find its coordinates:
$$
[\tilde v_j]_b=
w_j \left( [c_{(j,1)}]_b,\ldots, [c_{(j,k_j)}]_b,
[\tilde x_{i_{(j,1)}}^{b_{(j,1)}}]_b,
\ldots, [\tilde x_{i_{(j,l_j)}}^{b_{(j,l_j)}}]_b \right)=
$$
$$
=w_j \left( [c_{(j,1)}]_b,\ldots, [c_{(j,k_j)}]_b,
[\tilde x_{i_{(j,1)}}]_{bb_{(j,1)}^{-1}},
\ldots, [\tilde x_{i_{(j,l_j)}}]_{bb_{(j,l_j)}^{-1}} \right).
$$

Thereby, finding the solution $\{\tilde x_i\}_{i\in I}$ 
of $\{v_j=1\}_{j\in J}$ in $C=\prod\limits_{b \in \hat B}\hat H_b$
is the same as finding a set $\{[\tilde x_i]_b\}_{i\in I, b\in B}$
of elements of $\hat H$ such that
$$
w_j \left( [c_{(j,1)}]_b,\ldots, [c_{(j,k_j)}]_b,
[\tilde x_{i_{(j,1)}}]_{bb_{(j,1)}^{-1}},
\ldots, [\tilde x_{i_{(j,l_j)}}]_{bb_{(j,l_j)}^{-1}} \right) =1.
$$
for all $j\in J$ and $b \in \hat B$.

In other words, if the system of equations
$$
\left\{
w_j\left([c_{(j,1)}]_b,\ldots, [c_{(j,k_j)}]_b,
y_{i_{(j,1)},bb_{(j,1)}^{-1}},
\ldots,
y_{i_{(j,l_j)},bb_{(j,l_j)}^{-1}}\right)=1 \biggm|
j \in J, b \in \hat B
\right\}
$$
in variables
$\{ y_{i,b} \mid i \in I, b \in \hat B\}$
is solvable in $\hat H$,
then $\{v_j = 1\}$
has a solution consisting of elements of
$C$, hence, it is solvable in $\hat H \bar \wr \hat B$,
which proves the lemma.

Denote
$w_j \left( [c_{(j,1)}]_b,\ldots, [c_{(j,k_j)}]_b,
y_{i_{(j,1)},bb_{(j,1)}^{-1}},
\ldots,
y_{i_{(j,l_j)},bb_{(j,l_j)}^{-1}} \right)$
by $f_{j,b}$.
By the condition on the group $\hat H$,
if the system $\{f_{j,b}=1\}$ is
$p$-nonsingular,
then it is solvable $\hat H$ (which we need to prove).
Let us show this.

Consider $p$-trivializtion.
Let $m_{j,b}$ be the image of $f_{j,b}$ under $p$-trivialization.
There is a $\mathbb Z_p \hat B$-module structure on
$\displaystyle \sum\limits_{i\in I, b \in \hat B} \mathbb Z_p \cdot y_{i,b}$,
which is defined as:
$d \cdot y_{i,b} = y_{i,db}$
for $d \in \hat B$.
Thus, we get a free $\mathbb Z_p \hat B$-module
$\displaystyle \sum\limits_{i \in I}\mathbb Z_p\hat B \cdot y_{i,1}$.
Note that $m_{j,b}=b\cdot m_{j,1}$.

Consider a
linear combination of the rows $m_{j,b}$ which have the same index $j \in J$.
The linear combination $n_1m_{j,b_1} + \ldots + n_sm_{j,b_s}$
is equal to $(n_1b_1+\ldots+n_sb_s)\cdot m_{j,1}$,
i. e. $m_{j,1}$ multiplied by an element of group algebra
$\mathbb Z_p\hat B$.

Now consider a linear combination $\lambda$ of the rows $m_{j,b}$, where $j \in J$
can be different.
It is a sum of linear combinations of the
rows having the same index $j$.
Each of these combinations looks like $(n_1b_1+\ldots+n_tb_t)m_{j,1}$.
So $\lambda$ is a combination
of the rows $m_{j,1}$
with coefficients from
$\mathbb Z_p \hat B$
for some indices $j \in J$.
In particular, the rows $\{m_{j,b}\}_{i \in J, b \in \hat B}$
are linearly independent over $\mathbb Z_p$ if and only if
the rows $\{m_{j,1}\}_{j \in J}$ are independent over $\mathbb Z_p \hat B$.

Consider images of the rows $m_{j,1}$ under the natural ring homomorphism
$\varepsilon\colon \mathbb Z_p \hat B \to \mathbb Z_p$
which sends elements of $\hat B$ to the unity.
The image of the element from the $i$th place in the row $m_{j,1}$
(i. e. the image of the coefficient of $y_{i,1}$ in $m_{j,1}$)
is equal to the sum modulo $p$ of the exponent sums of $y_{i,b}$
in $f_{j,1}$ by all
$b\in\hat B$ while the index $i$ is fixed.
In other words, it is the sum
for all $b\in\hat B$ of exponent sums of $x_{i,b}$
(index $i$ is fixed)
in
$w_j\left([c_{(j,1)}]_1,\ldots, [c_{(j,k_j)}]_1,
y_{i_{(j,1)},b_{(j,1)}^{-1}},
\ldots,
y_{i_{(j,l_j)},b_{(j,l_j)}^{-1}}\right)$
modulo $p$.
And this is the same as the exponent sum of $x_i$ in the word
$v_j = w_j\left(c_{(j,1)},\ldots, c_{(j,k_j)}, x_{i_{(j,1)}}^{b_{(j,1)}},
\ldots, x_{i_{(j,l_j)}}^{b_{(j,l_j)}}\right)$
modulo $p$.
So the image of the row $m_{j,1}$ is the same as
the row of exponents sums of the variables $x_i$
in the word $v_j$ modulo $p$.
In particular, the images of the rows $m_{j,1}$
are linearly independent over $\mathbb Z_p$
as the system $\{v_j=1\}$ is $p$-nonsingular.
Hence, by Lemma \ref{Lemma1},
the rows $m_{j,1}$
are independent over $\mathbb Z_p \hat B$.

Thus,
$\{f_{j,b}=1\}$
is $p$-nonsingular.
It means that this system is solvable in $\hat H$,
which finishes the proof of the lemma.
\end{proof}

After Lemma \ref{Lemma2} is proved,
it can be applied to
$\tilde G_2 \bar \wr \hat B_1$ to get that
any $p$-nonsingular system of equations over this group
is solvable in $\tilde G_2 \bar \wr \hat B_1$ itself.
As this group contains $G$, we can take $\tilde G_2 \bar \wr \hat B_1$ as
$\hat G$.
The theorem's conclusion is true for $\hat G$:
it has a required subnormal series of length $n$
and any $p$-nonsingular system of equations over $\hat G$
is solvable in $\hat G$ itself.

The main theorem is now proved. \qed
\begin{rem}
If in the condition of the main theorem $n$ (the length of subnormal series)
is also equal to the derived length of $G$,
then every $p$-nonsingular system of equations over $G$
has a solution in a solvable group of the same derived length as $G$.
\end{rem}

\begin{rem}\label{strongrem}
An analogue of Lemma \ref{Lemma1}
in the case where $D$ is a torsion-free abelian group
and $\mathbb Z_p$ is replaced by $\mathbb Q$
is true as well.

Therefore, if we repeat the proof of the main theorem,
we can strengthen the result from \cite{KMR24}:
for any group $G$ with a subnormal series having torsion-free abelian factors
(the last factor can have torsion) there is a group $\tilde G \supset G$
with a similar subnormal series of the same length such that $\tilde G$ contains a
solution of every non-singular system of equations over $G$
(and even over $\tilde G$).

In \cite{KMR24} each non-singular system over $G$
could have a different solution group.
\end{rem}

\begin{rem}
An analogous result for two primes $p$ and $q$
(i.e. the main theorem for $\{p,q\}'$-torsion-free groups --- these are
group whose elements' finite orders are not divisible by primes
not from $\{p,q\}$ ---
instead of $p'$-torsion-free ones
and $\{p,q\}$-nonsingular systems ---
these are sumiltaneously $p$-nonsingular and
$q$-nonsingular systems --- instead of $p$-nonsingular ones)
is not true. This is shown by the example of order $42$ from
\cite[Proposition 1a]{KMR24}: the first factor of
this group's subnormal series
is $\{2,3\}'$-torsion-free, the equation over the group
is unimodular, however, the equation
has no solutions in metabelian groups.

Moreover, for every two prime numbers $p$ and $q$ we can construct
an example of a metabelian group with the
$\{p,q\}'$-torsion-free first factor such that
there is a unimodular equation over this group
which has no solutions in metabelian groups.

\begin{example}
Let $p,q$ be two different prime numbers.
Consider the metabelian group
$G=\langle c \rangle_2 \wr (\langle a \rangle_p \times \langle b \rangle_q)$.
Its order is $2^{pq}pq$.
Note that $c c^{ab} = [c,ab]$ belongs to the commutator subgroup of $G$.

As $p$ and $q$ are coprime, there
are integers $n$ and $m$ such that $np+mq=1$.
Then consider the equation
$$x^n\cdot x^{na}\cdot \ldots\cdot x^{na^{p-1}}\cdot
x^m\cdot x^{mb}\cdot\ldots\cdot x^{mb^{q-1}} = c c^{ab}$$
over $G$.
Since $np+mq=1$, this equation is unimodular.
And its right-hand side is an element of the commutator subgroup of $G$.

Assume that there is a metabelian group $\tilde G \supset G$
with a solution $\tilde x \in \tilde G$ of this equation.
After taking the quotient $\tilde G/\tilde G'$ the equality
$$\tilde x^n\cdot \tilde x^{na}\cdot \ldots\cdot \tilde x^{na^{p-1}}\cdot
\tilde x^m\cdot \tilde x^{mb}\cdot\ldots\cdot \tilde x^{mb^{q-1}} = c c^{ab}$$
transforms into $\tilde x = 1$.
It means that $\tilde x$ belongs to the commutator subgroup of $\tilde G$.
As $\tilde G$ is metabelian, this implies that $\tilde x$ commutes with its conjugates.

Using the notation $\tilde x^g \tilde x^h=\tilde x^{g+h}$
for simplicity, we get
$$
\left(\tilde x^n\cdot \tilde x^{na}\cdot \ldots\cdot \tilde x^{na^{p-1}}\cdot
\tilde x^m\cdot \tilde x^{mb}\cdot\ldots\cdot \tilde x^{mb^{q-1}}\right)^{1+ab}=
$$
$$
=\tilde x^{\left(n(1+b)\left(1+a+\ldots+a^{p-1}\right)+m(1+a)\left(1+b+\ldots+b^{q-1}\right) \right)}=
$$
$$
=\left(\tilde x^n\cdot \tilde x^{na}\cdot \ldots\cdot \tilde x^{na^{p-1}}\cdot
\tilde x^m\cdot \tilde x^{mb}\cdot\ldots\cdot \tilde x^{mb^{q-1}}\right)^{a+b}.
$$
However,
$$
\left(c c^{ab}\right)^{1+ab} = c^{1+2ab+a^2b^2}\neq
c^{a+b+a^2b+ab^2}=\left(c c^{ab}\right)^{a+b}.
$$
Thus, we get a contradiction:
some combinations of conjugates to the left-hand side of the equation are equal, though
the same combinations of conjugates to the right-hand side are not.

Therefore, the assumption of finding a solution
in a metabelian group is false.
\end{example}

\end{rem}

\begin{rem}
From the proof of the main theorem we can deduce that
$\hat G_n / \hat G_{n+1}$
is a Cartesian power of the divisible hull of the group $B_n=G_n/G_{n+1}$.
Therefore, if $B_n$ is $p'$-torsion-free (like the other factors
of the series of $G$), then
$\hat G_n / \hat G_{n+1}$
is $p'$-torsion-free as well.

More generally,
$\hat G_n / \hat G_{n+1}$ has the same set of prime orders of elements
as $B_n$.
For example, if $B_n$ is $\langle b \rangle_{30}$,
then $\hat G_n / \hat G_{n+1}$ has elements of prime orders
$2$, $3$ and $5$ and has no elements of prime orders $7$, $13$ or $19$.
\end{rem}

\section{The proof of minimality}\label{countersection}
In this section we call a solvable group,
over which there is a unimodular equation
that has no solutions in solvable groups of the
same derived length, a {\it counterexample}.
``Counterexample'' here means a counterexample to the statement
``any unimodular equation over a solvable group has a solution
in solvable groups with the same derived length''.

The main theorem shows that, if a metabelian group $G$
is an extension of an abelian group
by an abelian $p$-group for some prime $p$,
then any unimodular equation over $G$ has a solution
in some metabelian group, as unimodular equations are $p$-nonsingular
for every prime $p$.
If, moreover, $G$ is abelian, then any unimodular equation
has a solution in $G$ itself.
So in this case $G$ can not be counterexamples.

This helps us to prove that every
metabelian group of order at most $41$ is not a counterexample
and
thus the group from Proposition \ref{counterprop}
is indeed a minimal counterexample by order.

Consider metabelian groups of order at most $41$.

Firstly, note that, if the order of $G$ is $p^k$,
where $p$ is prime, then $G$ is nilpotent,
and unimodular equations over a nilpotent group are
solvable in this group itself by Shmel'kin's theorem \cite{Sh67}.
Hence, groups of order $p^k$ (even not necessarily metabelian)
are not counterexamples.
In our case, these are groups of orders
$1$, $2$, $3$, $4$, 
$5$, $7$,
$8$, $9$,
$11$, $13$,
$16$, $17$,
$19$, 
$23$, $25$, $27$,
$29$, $31$, $32$,
$37$ and $41$.

Further note that a non-abelian group of order $pq$,
where $p<q$ are two distinct primes, is (if it exists)
an extension of $\langle g \rangle_q$ by
$\langle h \rangle_p$,
for example, see \cite{Isa08}.
That is why groups of orders $6$,  
$10$, 
$14$, 
$15$,
$21$,
$22$, $26$,
$33$, $34$, $35$, 
$38$ and $39$ are also not counterexamples.

Also, there is a fact that in a group of order $p^2 q$,
where $p,q$ are distinct primes,
one of Sylow subgroups is normal
(this fact can also be seen in \cite{Isa08}).
It means that this group is either an extension of a group of order
$p^2$ (which is an abelian $p$-group) by
$\langle g \rangle_q$ or
an extension of $\langle g \rangle_q$
by abelian $p$-group.
In both cases this group is not a counterexample.
Thus, orders $12$, $18$, $20$ and $28$
do not need to be considered.

The remaining orders are framed in Table \ref{ordtable}.
These are $24$, $30$, $36$ and $40$.
\begin{table}[h]
\begin{center}
\large
\begin{tabular}{|c|c|c|c|c|c|c|}
\hline
1 &  2 & 3 & 4 &
5 &  6 & 7 \\ \hline
8 & 9 & 10 &
11 & 12 & 13 &14 \\ \hline
15 & 16 & 17 & 18
& 19 &
20 & 21 \\ \hline
22 & 23 & \fbox{24} & 25 & 26 & 27
& 28 \\ \hline
29 & \fbox{30} & 31 & 32
& 33 & 34 & 35 \\ \hline
\fbox{36} & 37 & 38 &  39 & \fbox{40} & 41 & \\ \hline
\end{tabular}
\end{center}
\caption{Considered and unconsidered orders of groups}\label{ordtable}
\end{table}

Suppose that metabelian group $G$
has one of these orders.
Consider a maximal abelian group $H$
containing the commutator subgroup of $G$.
Then $H$ will be normal with the abelian factor $G/H$
(as $H$ contains the commutator subgroup).
Note that the centralizer of $H$ in $G$
is precisely the group $H$, because, if an element $g$
lies in the cetnralizer of $H$ but does not lie in $H$,
then the subgroup $\langle g, H \rangle$
will also be abelian and will strictly contain $H$,
which is a contradiction to the maximality of $H$.
In other words, if we consider the action $G$
on $H$ by conjugation, then we get
an embedding of $G/H$ into $\operatorname{Aut}(H)$.

Now consider each of these orders separately.
\begin{itemize}
\item Order $24$:
$H$ can not have order $1$
or $2$ (as subgroups of these orders are central).
If $H$ has order $3$, $6$, $8$, $12$ or $24$, 
then $G/H$ is an abelian $p$-group.
(or $G$ is abelian itself). In this case
$G$ is not a counterexample.
The remaining case is $|H| = 4$.
In that case $H$ is either $\langle h \rangle_4$
or $V_4$.
Hence, $\operatorname{Aut}(H)$ is either
$\langle a \rangle_2$ or $S_3$.
As $G/H$ embeds into $\operatorname{Aut}(H)$
and has the order $6$, $\operatorname{Aut}(H)\cong S_3$.
But $G/H$ is a cyclic group of order $6$
(as it is abelian) and can not be embedded into $S_3$.
That is why the case $|H|=4$ is impossible.

We get that, if $|G| = 24$,
then $G$ is not a counterexample.
\item Order $30$:
again, $H$ can not have order
$1$ or $2$.
If $H$ is of order
$6$, $10$, $15$, then $G/H$ is an abelian $p$-group
(or $G$ is abelian itself).
If $|H|=3$ or $5$,
then $\operatorname{Aut}(H)$ has order either
$2$ or $4$, and $G/H$
(which has order $10$ or $6$)
can not be embedded into $\operatorname{Aut}(H)$.
So $|H|$ can not be $3$ or $5$.

Therefore, if $|G|=30$,
then $G$ is also not a counterexample.
\item Order $36$:
If $H$ has order
$4$, $9$, $12$, $18$ or $36$,
then $G/H$ is an abelian $p$-group
(or $G$ is abelian itself).
And $H$ can not have order
$1$, $2$ (as subgroups of these orders are central),
$3$ (because then $|\operatorname{Aut}(H)| = 2$ and $|G/H| = 18$)
or $6$ (because then $|\operatorname{Aut}(H)| = 2$
and $|G/H| = 6$).

Hence, in this case
$G$ is not a counterexample.
\item Order $40$:
If $H$ has order $5$, $8$, $10$, $20$ or $40$,
then $G/H$ is an abelian $p$-group
(or $G$ is abelian itself).
And
$H$ can not have order
$1$, $2$ (as subgroups of these orders are central)
or $4$ (because then $|\operatorname{Aut}(H)|=2$
or $6$, while $|G/H| = 10$).

So in this case $G$ is also not a counterexample.
\end{itemize}

As a result, a metabelian group of order at most $41$
can not be a counterexample. With the existence of a counterexample
of order $42$
we get the proof of Proposition \ref{counterprop}. \qed

\section{Open questions}\label{questsection}

\begin{quest}
Is there any example of a solvable group of derived length
greater than $2$ over which not every unimodular equation
has a solution in solvable groups of the same derived length?
\end{quest}

\begin{quest}
In particular, is there an example of a solvable group of order less than $42$
but of derived length greater than $2$, over which not every unimodular equation
has a solution in solvable groups of the same derived length?

For example, is $S_4$ (which is a group of order $24$ and derived length $3$)
such an example?
This group does not have a subnormal series satisfying the condition of the main theorem.
\end{quest}

\begin{quest}
Does every unimodular equation over the group from Proposition \ref{counterprop}
have a solution in a solvable group?
\end{quest}

\end{document}